\documentclass{amsart}

\newif\ifdraft\draftfalse
\newif\ifcite\citefalse
\newif\ifblow\blowtrue

\usepackage{amsfonts,amssymb, amsmath}
\usepackage{oldgerm}
\usepackage{amscd}
\usepackage[hyperindex]{hyperref}  
\usepackage[alphabetic]{amsrefs}
\ifdraft\ifcite\usepackage{showkeys}\else\usepackage[notcite,notref]{showkeys}\fi\fi

\newtheorem{proposition}[equation]{Proposition}
\newtheorem{theorem}[equation]{Theorem}

\newtheorem{lemma}[equation]{Lemma}

\theoremstyle{definition}

\theoremstyle{remark}

\newtheorem{example}[equation]{Example}

\numberwithin{equation}{section}

\def\bc{\begin{cases}}
\def\ec{\end{cases}}

\def\ol{\overline}

\def\a{\alpha}

\def\t{\tilde}

\def\ci{{\mathcal I}}

\def\cl{{\mathcal L}}

\def\cw{{\mathcal W}}

\def\bc{{\mathbb C}}

\def\br{{\mathbb R}}

\def\bz{{\mathbb Z}}

\def\er{\eqref}

\def\bz{\mathbb Z}

\def\br{\mathbb R}
\def\bc{\mathbb C}

\def\lp2{L_pH_{2p}}
\def\bean{\begin{eqnarray}}
\def\eean{\end{eqnarray}}
\def\bea{\begin{eqnarray*}}
\def\eea{\end{eqnarray*}}
\def\beq{\begin{equation}}
\def\eeq{\end{equation}}
\def\beq*{\begin{equation*}}
\def\eeq*{\end{equation*}}
\def\bal{\begin{align*}}
\def\eal{\end{align*}}
\def\baln{\begin{align}}
\def\ealn{\end{align}}

\def\beg{\begin{gather*}}
\def\eng{\end{gather*}}
\def\bqu{\begin{question}}
\def\equ{\end{question}}

\def\nm{\nonumber}

\def\ban{\begin{proof}[Answer]}
\def\ean{\end{proof}}

\def\ra{\Rightarrow}

\def\p{\partial}

\def\on{\operatorname}

\def\bqu{\begin{question}}
\def\equ{\end{question}}

\def\0110{\begin{matrix} 0 & 1\\1&0\end{matrix}}

\def\t{\tilde}

\def\fb{\mathfrak{b}}

\def\fg{\mathfrak{g}}
\def\fh{\mathfrak{h}}

\def\fl{\mathfrak{l}}

\def\fn{\mathfrak{n}}

\def\fs{\mathfrak{s}}

\def\ban{\begin{proof}[Answer]}
\def\ean{\end{proof}}
\def\wt{\widetilde}

\def\ben{\begin{equation}}
\def\een{\end{equation}}

\def\la{\langle}
\def\ra{\rangle}

\def\j1{{(j+1)}}

\def\ad{\text{ad}}

\begin{document}

\title[Toda field theories and 
standard differential systems]{Toda field theories and integral curves of standard differential systems}

\author{Zhaohu Nie}
\email{zhaohu.nie@usu.edu}

\address{Department of Mathematics and Statistics, Utah State University, Logan, UT 84322-3900}



\begin{abstract}
This paper establishes three relations between the Toda field theory associated to a simple Lie algebra and the integral curves of the standard differential system on the corresponding complete flag variety. 
The motivation comes from the viewpoint on the Toda field theories as Darboux integrable differential systems as developed in \cite{AFV}. 
First, we establish an isomorphism concerning regular functions on the jet space and on the unipotent subgroup in the setting of a simple Lie group. 
Using this result, we then show that in the sense of differential systems, after restricting one independent variable to a constant the Toda field theory becomes the system for integral curves of the standard differential system on a complete flag variety. Finally, we establish that, in terms of differential invariants, the Toda field theory is the quotient of the product of two such systems by a natural group action. 
\end{abstract}

\maketitle

\section{Introduction} 

	Let $\fg$ be a complex simple Lie algebra of rank $\ell$, and let  $(a_{ij})$ be its Cartan matrix. The (conformal) Toda field theory associated to $\fg$ is the following system of semilinear PDEs 	 
	\begin{equation}\label{toda}
	 u_{i, xy} = -\exp\Big(\sum_{j=1}^\ell a_{ij} u_j\Big),\quad 1\leq i\leq \ell,
	 \end{equation}
	where $x$ and $y$ are the independent variables and the $u_i$ are the unknown functions. 
	
	When $\fg=\fs\fl_2$, the corresponding Toda field theory is the ubiquitous Liouville equation. Toda field theories are fundamental examples of integrable systems and have been studied thoroughly in the literature (see the books \cites{LSbook, BBT}). They have zero curvature representations, and there have been
	detailed studies of their (local) solutions \cites{L, LS}. The solution structure has close relations to the representation theory of $\fg$ and its Lie group $G$ \cites{LS, KosToda}. In this regard, the author would also like to mention the recent classification results for the
	elliptic versions of the Toda systems \er{toda}, where one replaces $u_{i, xy}$ by $u_{i,z\bar z} = \frac{1}{4}\Delta u_i$ (see \cites{LWY, N15, LNW}). 
	In fact, the Liouville equation was first studied in this elliptic form in relation to conformal metrics of constant Gaussian curvature \cite{Liouville}. 
	
	In this paper, we are concerned with the aspects of Toda field theories from the viewpoint of differential systems. 
	By a differential system, we mean a Pfaffian system of 1-forms with constant rank on a manifold $M$ \cite{BCG3}. 
	In the dual viewpoint, the differential system is given by a subbundle of the tangent bundle $TM$, called a distribution, defined as 
	the kernel of the Pfaffian system. 
	We treat Toda field theories as Darboux integrable differential systems and 
	will establish the quotient structure from \cite{AFV}*{Theorem 1.4} for such systems. 
	We take a direct approach to establishing our results, and the proofs are crucially based on  
	the works of Kostant \cite{K3}, of Feigin and Frenkel \cites{FF-Inv, F5}, of Leznov and Saveliev
	\cites{LS, LSbook}, and of the author himself \cite{N14}. 
	
	A system of hyperbolic equations on the plane is called Darboux integrable if there is a  sufficient number of characteristic integrals, which are also called intermediate integrals in \cite{AFV}
	and integrals of motion in \cite{FF}. A \emph{characteristic integral}
	for the Toda field theory \er{toda} is a polynomial of the derivatives of the $u_i$ with respect to one independent variable whose derivative with respect to the other independent variable is zero if the $u_i$ are solutions. For example, for the Liouville
	equation $u_{xy} = -e^{2u}$, $I = u_{xx} - u_x^2$ is a characteristic integral on the $x$-side since $I_y=0$ for a solution $u$. Of course so is $u_{yy} - u_y^2$ on the $y$-side. 
	Such integrals have been intensively studied in the literature. We refer the reader to \cite{FF} for their cohomological interpretation and theoretical structure. 
	These characteristic integrals form the $\cw$-algebra and as such have been thoroughly studied in \cite{Feher1} and for more general gradings in \cite{Feher2}. 
	In \cite{N14}, the author  has directly established that for 
	the Toda field theory \er{toda} associated to a Lie algebra of rank $\ell$, there are $\ell$ basic characteristic integrals $I_j$ on one side.
	Furthermore, there is a completely algebraic and explicit 	algorithm for constructing them (see \er{bring up}). 
	This very algorithm will actually be important when we prove Theorem \ref{jets N} on the relation of jet spaces and unipotent subgroups. 
	In this connection, the author  would also like to mention the interesting application of characteristic integrals to the classification of solutions to elliptic 
	Toda systems with finite energy and with singular sources \cites{LWY, N15, LNW}. 
			
	Let us now briefly discuss the key results of this paper. One underlying theme of this paper is the relation of jet spaces and unipotent subgroups. 
	Such relations, especially for affine Lie algebras, lie at the heart of the soliton equations and can be said to be the reason for their existence. 
	This approach was thoroughly developed in a series of papers by Feigin, Frenkel and Enriquez \cites{FF, FF-Inv, EF} and was beautifully surveyed in \cite{F5}. 
	The following Theorem \ref{jets N} is an analogue of \cite{FF-Inv}*{Prop. 4} (see also \cite{F5}*{Theorem 1.1}) in the setting of 
	finite-dimensional simple Lie groups. 
	This theorem is proved in Section \ref{jets}. 
	
	We first introduce some notation. 
	Let $\fh$ be a fixed Cartan subalgebra of $\fg$. Let $\Delta = \Delta^+ \cup \Delta^-$ be a decomposition of the set of roots of $\fg$ into the sets of positive and 
	negative ones, and let $\{\a_1, \dots, \a_\ell\}$ be the set of positive simple roots. 
	Let $e_\a$ be a root vector in the root space $\fg_\a$ for $\a\in \Delta$. We will specify our normalization of the $e_\a$ when needed.  	
	Let $G$ be any connected Lie group integrating $\fg$, and let $B_+$ and $N_-$ be the Borel and lower unipotent subgroups corresponding to the Lie
	subalgebras $\fb_+ = \fh \oplus \oplus_{\a\in \Delta^+} \fg_\a$ and $\fn_- = \oplus_{\alpha\in \Delta^-} \fg_\a$. Then $N_-$ is diffeomorphic to a Euclidean space
	\cite{Knapp}. 
	
	The homogeneous space $G/B_+$ is called a complete flag variety, and it has a natural transitive left $G$-action by $g(k B_+) = gk B_+$ for $g\in G$ 
	and $k B_+\in G/B_+$. 
	It is well-known that
	the composition $N_-\hookrightarrow G\to G/B_+$ of the inclusion and the projection 
	is injective and its image is a big cell
	(see e.g. \cites{KosToda, LSbook, FF-Inv}). 
	We will mostly work with this big cell  in the local pictures and continue to denote it by $N_-$. 
	In particular, $N_-$ has a $\fg$-action induced from the above left $G$-action. 
	
	\begin{theorem}\label{jets N} Let $\bc[{\bf U}]$ be the ring of polynomials on the variables $u_{i, x}^{(n)}$ for $1\leq i\leq \ell,\ n\geq 1$
	equipped with an action of the derivative $\p_x$ which sends $u_{i,x}^{(n)}$ to $u_{i, x}^{(n+1)}$. 
	Let $\ci$ be the ideal of $\bc[{\bf U}]$ generated by the characteristic integrals of the Toda field theory \er{toda} on the  $x$-side 
	and their derivatives $\p_x^m I_j$ for $1\leq j\leq \ell,\ m\geq 0$. 
	Moreover, let $\bc[N_-]$ be the ring of regular functions on $N_-$. 
	Then we have a canonical isomorphism 
	$$\bc[{\bf U}]/\ci \cong \bc[N_-],$$
	under which 
	the derivative $\p_x$ on the left is identified with the derivation $\cl_e$ on the right for the infinitesimal action of the principal nilpotent element $e=\sum_{i=1}^\ell e_{\a_i} \in \fg$. 
	\end{theorem}
	
	By \cite{AFV}*{Theorem 1.4}, a Darboux integrable differential system is the quotient, in the sense of differential systems, of the product of two Pfaffian systems by the action of a common Lie group $G$. One of the main goals of this 
	paper is to explicitly demonstrate this for the Toda field theory \er{toda} in a direct way, and this is achieved in our Theorem \ref{main}. First we need to introduce the two Pfaffian systems in our result. 
	
	There exists a so-called grading element 
	$H_0\in \fh$ such that $\a_i(H_0) = 1, \ \forall 1\leq i\leq \ell$. 
	The principal grading of $\fg$ is 
	\begin{equation}\label{prin grad}
	\fg = \bigoplus_i \fg_i, \quad\text{where }\fg_i = \{ X\in \fg\,|\, [H_0, X]= i X\}.
	\end{equation}
	Identify the tangent space of $G/B_+$ at the point $o=1\cdot B_+$ with $\fg/\fb_+$, where $1\in G$ is the identity element.
	Define the distribution $D\subset T(G/B_+)$ as the $G$-invariant distribution equal to $\fg_{-1}$ mod $\fb_+$ at $o$. 
	Clearly $\fg_{-1} = \oplus_{i=1}^\ell \fg_{-\a_i}$ is the direct sum of negative simple root spaces.
	The Pfaffian system dual to the distribution $D$ is called the
	\emph{standard differential system} for the principal grading \cite{Yam}, and we call $D$ the \emph{standard distribution}.
	 For an interval $I\subset \br^1$, a curve $\Phi : I\to G/B_+$ is called an \emph{integral curve of the standard differential system} if 
	 \begin{equation}\label{stand sys}
	 \Phi'(y) = d\Phi\Big(\frac{\p}{\p y}\Big)\in D_{\Phi(y)}, \quad \forall y\in I
	 \end{equation}
	 (see \cite{DZ} for geometric studies of such curves). 
	 Again, we mostly work with $N_-\subset G/B_+$ which has a standard differential system naturally induced by restriction. 
	
	Our second main result concerns the differential system representations of Toda field theories.  
	The proof is contained in Section \ref{sect-rest} and uses Theorem \ref{jets N} in an essential way. 
	\begin{theorem}\label{Serre} In the differential system for the Toda field theory \er{toda}, if we restrict $x$ to a constant and discard jet prolongations 
	(see \er{with x-jets} and the discussion after it), then we obtain the differential system on the first jet space $J^1_y(N_-, D)$ of 
	the integral curves to the standard differential system on the unipotent Lie group $N_-$ in the following sense: 
	there is an isomorphism between the global vector field generators 
	for the corresponding distributions which respects all Lie brackets. 
	\end{theorem}
	
	The above theorem can be used to determine the two Pfaffian systems in \cite{AFV}*{Theorem 1.4} with the other one obtained 
	by restricting $y$ to a constant. 
	Therefore, it is natural to consider the quotient of the product of two such systems by the diagonal $G$-action. For differential systems with symmetry, 
	quotients are defined in terms of 
	differential invariants. 
	Our third main result studies such differential invariants, and it is proved in Section \ref{sect-invs}.
	
	\begin{theorem}\label{main}
	Let $J^1_y(N_-, D)$ denote the 1st jet space of the integral curves to the standard distribution $D$ on $N_-$ for the independent variable $y$. 
	For the other independent variable $x$, consider the similarly defined $J^1_x(N_+, \t D)$ (see the paragraph containing \er{def y-curve}). 	
		Then there exist $\ell$ differential invariants of the prolonged natural $G$-action on the product 
		$J^1_x(N_+, \t D)\times J^1_y(N_-, D)$,
		which satisfy the Toda field theory \er{toda}. 
	\end{theorem}

	\medskip
	\noindent{\bf Acknowledgment.} I express my deep gratitude to Prof. Ian Anderson for introducing this topic to me, and for generously sharing many deep insights and 
	helpful discussions throughout the 
	course of the last several years. Prof. Anderson's several Maple programs are also very useful in carrying out the computations in this paper. 
	I thank an anonymous referee for thorough readings, for pointing out missing arguments and inaccuracies in earlier versions, 
	and for many constructive suggestions, 
	which significantly improved the quality of this paper in terms of both mathematics and exposition. 
	
	\section{Jet spaces and unipotent subgroups}\label{jets}
	
	In this section, we prove Theorem \ref{jets N}, which will be used in a crucial way 
	in proving Theorem \ref{Serre}. Theorem  \ref{jets N} 
	has its own interest and is the analogue of \cite{FF-Inv}*{Prop. 4}, also stated as \cite{F5}*{Theorem 1.1}, in the setting of simple Lie groups. 
	Interestingly in this finite-dimensional case, the characteristic integrals naturally come up. Now we recall the construction of such integrals from \cite{N14}. 
	
	Let us introduce the zero curvature representation of \er{toda}. 
	We normalize the root vectors $e_{\a_i}$ and $e_{-\a_i}$ for $1\leq i\leq \ell$ 
	such that $\a_i(H_{\a_i})=2$, 
	where $H_{\a_i} = [e_{\a_i}, e_{-\a_i}]$. Then by \cite{FH}*{p. 208}, 
	\begin{equation}\label{another Cartan}
	\a_i (H_{\a_j}) = a_{ij}, \quad 1\leq i, j\leq \ell.
	\end{equation} 
	Let 
\begin{equation}\label{players}
{\bf u}=\sum_{i=1}^\ell u_{i,x} H_{\a_i}, \quad
e=\sum_{i=1}^\ell e_{\a_i}, \quad
Y=\sum_{i=1}^\ell e^{\rho_i} e_{-\a_i},
\end{equation}
where throughout the paper we use the shorthand notation 
\begin{equation}\label{rhoi}
\rho_i=\sum_{j=1}^\ell a_{ij} u_j. 
\end{equation}
Then the Toda field theory \er{toda} is equivalent to the following zero curvature equation
\begin{equation}\label{zero curv}
[\p_x+e+{\bf u},\, \p_y-Y]=0.
\end{equation}

	With respect to the principal grading \er{prin grad} of $\fg$, let $\fs$ be a homogeneous complement of $[e,\fg]$ in $\fg$, that is, 
\begin{equation}\label{split}
\fg\cong \fs\oplus [e,\fg].
\end{equation}
Then by \cite{K2}, $\fs\subset \fn_-$ and $\dim \fs=\ell$. We call $\fs$ a Kostant slice, and let $\{s_j\}_{j=1}^\ell$ be a homogeneous basis of $\fs$
with nonincreasing principal gradings $-m_j$. The $m_j$ are called the \emph{exponents} of the Lie algebra $\fg$. 

By \cite{DS} and \cite{N14}*{Remark 2.1} (see also the proof of Proposition \ref{polynm}), we can bring the first element in \er{zero curv} into its Drinfeld-Sokolov gauge in a unique way. That is, there exists a unique element $M \in N_-$ (whose entries are differential polynomials of the $u_i$) such that 
\begin{equation}\label{bring up}
M (\p_x+e+{\bf u}) M^{-1}=\p_x + e + {\bf I},\quad {\bf I}=\sum_{j=1}^\ell I_j s_j\in \fs, 
\end{equation}
where the $I_j$ are differential polynomials of the $u_i$. 
The uniqueness of $M$ is easily proved by induction on the principal grading \er{prin grad} using the fact 
that 
	\begin{equation*}\label{kern info}
	\ker \on{ad}_e \,\cap\, (\fh \oplus \fn_-) = 0.
	\end{equation*}

	Then Theorem 2.1 in \cite{N14} proves directly that the $I_j$ are the basic characteristic integrals of the Toda field theory \er{toda} on the $x$-side. See also \cite{FF}*{Prop. 2.4.7 and \S 2.4.1}. 
	
	We first show the following. 
	\begin{proposition}\label{polynm}
	There exists a set of generators for 
	the ring $\bc[{\bf U}]/\ci$ 
	whose cardinality is $\dim N_-$. 
	\end{proposition}

	
	\begin{proof}
	For a differential monomial in the $u_i$, we call by its \emph{degree} the sum of the orders of differentiation multiplied by the algebraic degrees of the corresponding factors. 
	For example $I = u_{xx} - u_{x}^2$ for the Liouville equation has a homogeneous degree $2$. 
	According to \cite{FF}*{Prop. 2.4.7}, the degrees $d_j$ of the homogeneous characteristic integrals $I_j$ are 
	$d_j = m_j + 1$, and this can also be seen from the algorithm 
	\er{bring up}. 
	
	Define 
	$$
	s_j^k = (-\ad_e)^k s_j = [\cdots[s_j, \underbrace{e], \cdots, e]}_k,\quad 0\leq k\leq 2m_j,\ 1\leq j\leq \ell. 
	$$
	It is known from \cite{K2} that the above $s_j^k$ form a basis of $\fg$. 
	Hence $\{H_{\a_i}\}_{i=1}^\ell$ and $\{s_j^{m_j}\}_{j=1}^\ell$ are two bases of $\fh$. We define the matrix $(c_{ji})$ by 
	$$
	H_{\a_i} = \sum_{j=1}^\ell c_{ji} s_j^{m_j}, \quad 1\leq i\leq \ell.
	$$
	Clearly $(c_{ji})$ is nondegenerate. 
	Now we show that 
	\begin{equation}\label{lot}
	I_j = \sum_{i=1}^\ell c_{ji} u_{i, x}^{(d_j)} + l.o.t.,
	\end{equation}
	where l.o.t. stands for terms which are products of lower order derivatives. 
	
	
	We will use this opportunity first to give more details on the existence and the uniqueness of $M$ in \er{bring up}, which satisfies
\begin{equation}\label{bring out}
-\partial_x M \cdot M^{-1} + M (e + {\bf u}) M^{-1} = e + {\bf I}.
\end{equation}
It is well-known that the exponential map $\exp: \fn_- \to N_-$ is a diffeomorphism \cite{Knapp}*{Cor. 1.126, Thm. 6.46}. Using the so-called coordinates of the second kind \cite{Knapp}*{p. 76}, write 
\begin{equation}\label{write g}
M = e^{a_1}\cdots e^{a_{m_\ell}}, \quad a_i\in \fg_{-i},\ i=1,\dots,m_\ell,
\end{equation}
where $m_\ell$ is the largest exponent of $\fg$. 
We will uniquely determine the $a_i$ inductively. 

For $i\geq 1$, consider $M_{i-1} = e^{a_1}\cdots e^{a_{i-1}}$ with the convention $M_0 = 1$. 
Define
$$
L_{i-1} := -\partial_x M_{i-1} \cdot M^{-1}_{i-1} + M_{i-1} (e + {\bf u}) M^{-1}_{i-1}.
$$ 
In general, for $X\in \fn_-$, let $X_j$ denote its component in $\fg_{-j}$ in the principal grading \er{prin grad}. 
Inductively, assume that 
$$
(L_{i-1})_j \in \fs\cap \fg_{-j}\quad \text{for }0\leq j\leq i-2.
$$
Note that this inductive hypothesis is vacuous when $i=1$. 
By \er{split}, the component of $L_{i-1}$ in $\fg_{-(i-1)}$ can be uniquely written as 
	\begin{equation}\label{write i}
	(L_{i-1})_{i-1} = [e, a_i] + {\bf J}_{i-1}
	\end{equation}
	with $a_i\in \fg_{-i}$ and ${\bf J}_{i-1} \in \fs\cap \fg_{-(i-1)}$.
	
	For a general $\t a_i\in g_{-i}$ and with 
	$\t M_i = M_{i-1} e^{\t a_i}$, we have 
	\begin{equation}\label{comp li}
	\begin{split}
	\t L_{i} &=-\p_x \t M_i \cdot \t M_i^{-1} + \t M_i (e + {\bf u}) \t M_i^{-1}\\
	&=-\p_x M_{i-1} \cdot M_{i-1}^{-1} - M_{i-1} (\p_x  e^{\t a_i}\cdot e^{-\t a_i}) M_{i-1}^{-1} + M_{i-1}e^{\t a_i} (e+{\bf u}) e^{-\t a_i} M_{i-1}^{-1}.
	\end{split}
	\end{equation}
	Since $\t a_i\in \fg_{-i}$, we see that 
	\begin{align}
	(\t L_i)_j &= (L_{i-1})_j\in \fs\cap \fg_{-j},\quad 0\leq j\leq i-2,\label{keep}\\
	(\t L_i)_{i-1} &= (L_{i-1})_{i-1} + [\t a_i, e]. \nm
	\end{align}
	If and only if we choose $\t a_i = a_i$ from \er{write i}, we see that 
	\begin{equation}\label{lii1}
	(L_i)_{i-1} = (L_{i-1})_{i-1} + [a_i, e] = {\bf J}_{i-1} \in \fs \cap \fg_{-(i-1)},
	\end{equation}
	where $L_i$ denotes $\t L_i$ with $\t a_i = a_i$. 
	Therefore, 
	$$
	(L_i)_j\in \fs\cap \fg_{-j}\quad \text{for }0\leq j\leq i-1.
	$$
	The inductive proof for the existence and uniqueness of $M$ is completed. 
	
	Furthermore, in view of \er{bring out}, \er{write g},  \er{keep}, \er{lii1} and \er{write i}, we have that 
	\begin{equation}\label{stay}
	{\bf I}_{i-1} = (L_{m_\ell})_{i-1} = (L_i)_{i-1} = \text{component of }(L_{i-1})_{i-1}\text{ in }\fs\cap \fg_{i-1}.
	\end{equation}
	That is, the component 
	of  $(L_{i-1})_{i-1}$  in $\fs\cap \fg_{i-1}$ is equal to the component of ${\bf I}$ from \er{bring out} in $\fg_{-(i-1)}$. 
	
	We now show that 
	the component of $L_i$ in $\fg_{-i}$ is
	\begin{equation}\label{lii}
	(L_i)_i = -\p_x a_i + l.o.t.
	\end{equation}
	by showing that the other terms in $(L_i)_i$ from \er{comp li} contain only products of lower order derivatives. 
	We have that 
	$$
	-\p_x M_{i-1}\cdot M_{i-1}^{-1} = - \sum_{j=1}^{i-1} M_{j-1} (\p_x e^{a_j}\cdot e^{-a_j}) M_{j-1}^{-1}. 
	$$
	Since $a_j\in g_{-j}$ and $1\leq j\leq i-1$, we see that all the terms in $(-\p_x M_{i-1}\cdot M_{i-1}^{-1})_i$ contain only products of lower order derivatives. 
	The same argument also applies to the last element $M_{i-1}e^{a_i} (e+{\bf u}) e^{-a_i} M_{i-1}^{-1}$. 
	Since $a_i\in \fg_{-i}$, we see that 
	$$(- M_{i-1} (\p_x  e^{a_i}\cdot e^{-a_i}) M_{i-1}^{-1})_i = -\p_x a_i.$$
	Actually, $ -\p_x a_i$ may contain products too, but it is the only term that contains non-products. 
	
	Now we consider more specifically the highest order derivative terms. Continuing with the above notation, we have 
	$$
	(L_0)_0 = (e + {\bf u})_0 = \sum_{i=1}^\ell u_{i, x} H_{\a_i} = \sum_{j, i=1}^\ell c_{ji} u_{i, x} s_j^{m_j}.
	$$
	From \er{write i}, we have
	$$
	a_1 = -\sum_{j, i=1}^\ell  c_{ji} u_{i, x} s_j^{m_j-1}.
	$$
	By \er{lii}, we have
	$$
	(L_1)_1 = -\p_x a_1 + l.o.t. = \sum_{j, i=1}^\ell  c_{ji} u_{i, x}^{(2)} s_j^{m_j-1}+ l.o.t..
	$$
	For all simple Lie algebras, $m_1 = 1$ and $m_2>1$ by \cite{K1}. Therefore, by \er{stay} we see that $I_1 = \sum_{i=1}^\ell c_{1i} u_{i, x}^{(2)} + l.o.t.$. 
	
	Continuing this way, we see that \er{lot} holds for all $j=1,\dots, \ell$. 
	
	Then through the Gaussian elimination,
	there exists a permutation $\sigma$ in the symmetric group $S_\ell$
	such that the following set
	\begin{equation}\label{gens}
	\{u_{\sigma(i), x}^{(k)}\,|\,1\leq k\leq m_i,\ 1\leq i\leq \ell\}
	\end{equation}
	is a set of generators of $\bc[{\bf U}]/\ci$ in the sense that the other 
	$$
	u_{\sigma(i), x}^{(k)},\quad k\geq d_i,\ 1\leq i\leq \ell
	$$
	are solved as polynomials in them modulo $\ci$. 
	The cardinality of the set \er{gens} is 
	$$
	m_1 + \cdots + m_\ell = \dim N_-
	$$
	by \cite{K1}. 
	\end{proof}

	We now study the restriction of the left $G$-action on $G/B_+$ to $N_-$. 
	The natural multiplication map $N_-\times B_+\to G$ is injective and its image $G_r$ is an open subset, called the regular part of $G$ (see 
	\cite{KosToda}*{Eq. (2.4.4)}, \cite{LSbook}*{Eq. (1.5.6)}). 
	For $n_1\in N_-$, 
	there exists an open set $1\in U\subset G$ 
	such that $gn_1\in G_r$ for $g\in U$. 
	The action $g\in U$ on $n_1$ is obtained by the following normalization procedure:
	\begin{equation}\label{action}
		\text{if } gn_1 = \t n_1 p\in N_-B_+, \text{ then } g\cdot n_1 := \t n_1.
	\end{equation}
	For $a\in \fg$, we denote by $\cl_a$ 
	the infinitesimal action of $a$ 
	on $\bc[N_-]$. 
	Explicitly, 
	\begin{equation}\label{la}
	(\cl_a f)(n) = \frac{d}{dt}\Big|_{t=0} f(\exp(-ta)\cdot n),\quad n\in N_-,\ f\in \bc[N_-].
	\end{equation}
	In particular, we have the vector field $\cl_e$ for $e=\sum_{i=1}^\ell e_{\a_i}$ defined on $N_-$. 
	
	Following \cite{F5}*{Lemma 1.1}, we have the following lemma about the action of $\cl_a$. 
	Here we take a faithful representation of $N_-$ and represent an element $K\in N_-$ by a matrix whose $(i, j)$th entry $f_{ij}$ is considered as a regular function in $\bc[N_-]$. 
	Then $\cl_a K$ is the matrix whose $(i, j)$th entry is $\cl_a f_{ij}$. 
	\begin{lemma} For $a\in \fg$ and $K\in N_-$, we have 
	\begin{equation}\label{reason}
	K^{-1} \cl_a K = -(K^{-1} a K)_-,
	\end{equation}
	where $(\cdot)_- : \fg = \fn_- \oplus \fb_+ \to \fn_-$ is the projection.
	\end{lemma}
	
	\begin{proof} Choose a one-parameter subgroup $a(t)$ of $G$ such that $a(t) = 1 - t a + o(t)$. We have 
	$
	a(t) K = K - t a K + o(t).
	$
	For small $t$, we can factor $a(t) K$ into a product $L_- L_+$, where $L_- = K + t L_-^{(1)} + o(t) \in N_-$ and $L_+ = 1 + t L_+^{(1)} + o(t) \in B_+$. Therefore we have
	$$
	K L_+^{(1)} + L_-^{(1)} = -aK.
	$$
	From this we see that $L_-^{(1)} = - K(K^{-1}aK)_-$. This proves
	formula \er{reason}.
	\end{proof}
	
	For $1\leq i\leq \ell$, let $\omega_i \in \fh^*$ be the $i$th fundamental weight defined by 
	the conditions that 
	$\omega_i(H_{\a_j}) = \delta_{ij}$ for $1\leq j\leq \ell$. 

	\begin{proof}[Proof of Theorem \ref{jets N}]
	Let $K\in N_-$. Following Kostant \cite{K3} and \cites{FF-Inv, F5}, we define the following functions on $N_-$
	\begin{equation}\label{coord}
	\begin{split}
	v_i &:= ( \omega_i, K^{-1} e K), \quad 1\leq i\leq \ell,\\
	v_i^{(n)} &:= \cl_e^n v_i, \qquad\quad\quad\ n\geq 0.
	\end{split}
	\end{equation}
	Here $(\cdot, \cdot)$ is the Killing form on $\fg$, and $\omega_i$ is regarded as an element of $\fh$ via the Killing form. 
	For simplicity, we write $u_i^{(k)}$ for $u_{i, x}^{(k)}$. 
	Define a ring homomorphism
	\begin{equation}\label{whichmap}
	\varphi: \bc[{\bf U}] \to \bc[N_-], \ \ u_i^{(k)}\mapsto v_i^{(k-1)},\ k\geq 1.
	\end{equation}
	We note that under this homomorphism, $\p_x$ corresponds to $\cl_e$ by the above definitions. 
	Now we show that the images $\varphi(I_j)$ of the characteristic integrals $I_j$ are zero. That is, if we replace the $u_i^{(k)}$ by the $v_i^{(k-1)}$ in $I_j$, we 
	get zero. 
	
	We find the images $\varphi(I_j)$ by 
	adapting our algorithm in \er{bring up} for computing $I_j$. 
	So we replace $\p_x$ by $\cl_e$ and ${\bf u}$ by ${\bf v}=\sum_{i=1}^\ell v_i H_{\a_i}$. Then by the same
	reason as for \er{bring up}, there exists a unique element $M_1\in N_-$, whose entries are functions on $N_-$, such that 
	\begin{equation}\label{theory}
	M_1 (\cl_e + e + {\bf v}) M_1^{-1} = \cl_e + e + \sum_{j=1}^\ell J_j s_j.
	\end{equation}
	Then $J_j = \varphi(I_j)$ for $1\leq j\leq \ell$. 
	
	We will show below that 
	the unique $M_1$ in \er{theory} is the $K$ in \er{coord}, 
	and \er{theory} becomes
	\begin{equation}\label{simple}
	K (\cl_e + e + {\bf v}) K^{-1} = \cl_e + e. 
	\end{equation}
	Therefore, $\varphi(I_j)=J_j = 0$ for $1\leq j\leq \ell$. 
	
	It is clear that $K(\cl_e K^{-1}) = -(\cl_e K) K^{-1}$. Therefore by \er{reason}, we have
	\begin{align*}
	K (\cl_e + e + {\bf v}) K^{-1} &= \cl_e - (\cl_e K) K^{-1} + K(e + {\bf v}) K^{-1} \\ 
	&= \cl_e +  K (K^{-1} e K)_- K^{-1} + K(e + {\bf v}) K^{-1}\\
	&= \cl_e +  K \big((K^{-1} e K)_- + e + {\bf v}\big) K^{-1} \\
	&= \cl_e +  K(K^{-1} e K) K^{-1} \\
	&= \cl_e +  e.
	\end{align*}
	Here we have used that $K^{-1} e K 
	= e + {\bf v}+ (K^{-1} e K)_-$ by \er{coord} since the Cartan component of 
	$K^{-1} e K$ is equal to $\sum_{i=1}^\ell (\omega_i, K^{-1}eK) H_{\a_i} = {\bf v}$. 

	Furthermore, since $\p_x$ corresponds to $\cl_e$, we have $\varphi(\p_x^m I_j) = \cl_e^m J_j = 0$ for $m\geq 0$. 
	Therefore, the morphism $\varphi$ in \er{whichmap} descends to a morphism, which we continue to denote by the same notation, 
	\begin{equation}\label{descend}
	\varphi: \bc[{\bf U}]/\ci \to \bc[N_-].
	\end{equation}
	Now we show that this morphism is an isomorphism, and this part is analogous to the proof of Theorem 1.1 in \cite{F5}. 
	
	To prove that $\varphi$ is injective, we have to show that the images $\varphi(u_{\sigma(i)}^{(k)}) = v_{\sigma(i)}^{(k-1)}$ for $1\leq k\leq m_i,\ 1\leq i\leq \ell$ 
	of the set \er{gens} are algebraically independent. We will do that by showing that at the identity $1\in N_-$, the $dv_{\sigma(i)}^{(k-1)}\big|_1$ for 
	$1\leq k\leq m_i,\ 1\leq i\leq \ell$ are linearly independent in the 
	 cotangent space $T^*_1 N_-$, which 
	is identified with $\fn_+$ through the Killing form. 
	It can be checked that $dv_i\big|_1 = e_{-\a_i}^*$, which is identified with $\frac{(\a_i, \a_i)}{2}e_{\a_i}$ 
	since $(e_{-\a_i}, e_{\a_i})=\frac{2}{(\a_i,\a_i)}$ by our normalization \er{another Cartan}. 
	Furthermore by \er{coord}, we have $dv_i^{(n)}\big|_1 = \on{ad}_e^{n} (dv_i\big|_1)$. 
	The set 
	$$
	\{ \on{ad}_e^{k-1}(dv_{\sigma(i)}|_1)\,|\, 1\leq k\leq m_i,\ 1\leq i\leq \ell\}
	$$
	is a basis of $\fn_+$ by \cite{K1} and hence is linearly independent. 
		
	To prove that $\varphi$ is surjective, we introduce gradings and compare them. 
	On $\bc[{\bf U}]/\ci$, we set $\deg u_{\sigma(i)}^{(k)} = k$ for $1\leq k\leq m_i,\ 1\leq i\leq \ell$ . 
	The above algebraic independence also shows that $\bc[{\bf U}]/\ci$ is a polynomial algebra with 
	$b_i$ generators of degree $i$ by \cite{K1}, where $b_i = \dim \fg_i$ from \er{prin grad} for $1\leq i\leq m_\ell$. 
	
	On $\bc[N_-]$, we take the derivation $\cl_{H_0}$ as the grading operator, where $H_0$ is the grading element for the principal grading \er{prin grad}. 
	Then it can be checked that $\deg v_i = 1$. Since $[H_0, e]=e$, the degree of $\cl_e$ also equals 1 and so $\deg v_i^{(k-1)} = k$. 
	Therefore, the morphism $\varphi$ in \er{descend} preserves degrees. Since $\exp: \fn_-\to N_-$ is an isomorphism, we see that $\bc[N_-]$ is also a polynomial algebra on $b_i$ 
	generators of degree $i$ for $1\leq i\leq m_\ell$. 
	Hence the character of $\bc[N_-]$ coincides with the character of $\bc[{\bf U}]/\ci$, and $\varphi$ is an isomorphism. 
	Here by a character of a $\bz$-graded vector space $V$ we understand the formal power series
	$$
	\on{ch}V = \sum_{n\in \bz} \dim V_n q^n,
	$$
	where $V_n$ is the homogeneous subspace of $V$ of degree $n$. Concretely, 
	$$
	\on{ch} \bc[{\bf U}]/\ci = \on{ch} \bc[N_-] = \prod_{i=1}^{m_\ell} (1-q^i)^{-b_i}.
	$$
	\end{proof}
		
	\begin{example}\label{A_2} 
	To illustrate our results, let us
	consider the example of $A_2$ Toda field theory \er{toda} as
	$$
	\begin{cases}
	u_{1, xy} = - e^{2u_1-u_2}\\
	u_{2, xy} = - e^{-u_1+2u_2}.
	\end{cases}
	$$
	Let $E_{ij}$ denote the matrix of dimension 3 whose only nonzero element is 1 at position $(i, j)$. 
	We use standard choices of $e = E_{12} + E_{23}$, $H_{\a_1} = E_{11}- E_{22}$, $H_{\a_2} = E_{22} - E_{33}$, 
	$s_1 = E_{21}$, and $s_2=E_{31}$. Then  the transformation matrix from $(H_{\a_1}, H_{\a_2})$ to $(s_1^1, s_2^2)$
	is $\left(\begin{smallmatrix} -1 & -1\\ 0 & -1\end{smallmatrix}\right)$. 
	
	Our formula from \er{bring up} computes that 
	\begin{equation}\label{the Is}
	\begin{split}
	I_1 &= -u_{1, xx} - u_{2, xx} + u_{1, x}^2 - u_{1, x} u_{2, x} + u_{2, x}^2,\\
	I_2 &= -u_{2, xxx} + 2u_{2, xx}u_{2, x} - u_{1, xx}u_{2, x} + u_{1, x}^2 u_{2, x} - u_{1, x}u_{2, x}^2.
	\end{split}
	\end{equation}
	Note that \er{lot} is thus checked. 
	
	From these we see that a set of ring generators for $\bc[{\bf U}]/\ci$ in Proposition \ref{polynm} can be chosen as $(u_{1, x}, u_{2, x}, u_{2, xx})$
	 or $(u_{1, x}, u_{2, x}, u_{1, xx})$. 
	
	Let us also show the content of Theorem \ref{jets N} using this example. An element $K\in N_-$ has the following form
	\begin{equation}\label{def K}
	K = \begin{pmatrix}
	1 & & \\
	v_1 & 1 & \\
	v_3 & v_2 & 1
	\end{pmatrix},
	\end{equation}
	where $v_1, v_2, v_3$ are the coordinates on $N_-$. 
	We see that the definitions of $v_1$ and $v_2$ in \er{coord} are compatible with the above. 
	By \er{la} and the action \er{action}, we compute that 
	$$
	\cl_e = (v_1^2 - v_3) \frac{\p}{\p v_1} + (v_2^2 + v_3 - v_1 v_2)\frac{\p}{\p v_2} + v_1 v_3\frac{\p}{\p v_3}.
	$$
	
	It is easy to check directly  that for the characteristic integrals  in \er{the Is} and the $\varphi$ in \er{whichmap}, we have $\varphi(I_1)=0$ and $\varphi(I_2) =0$. 
	Then the map in \er{descend} is 
	\begin{equation}\label{transf}
	\begin{split}
	\varphi : \bc[{\bf U}]/\ci \to \bc[N_-];&\ u_{1, x}\mapsto v_1,\ u_{2, x}\mapsto v_2,\\
	&\ u_{2, xx} \mapsto v_2^{(1)} = \cl_e v_2 = v_2^2 + v_3 - v_1 v_2.
	\end{split}
	\end{equation}
	\end{example}
	
\section{From Toda field theories to standard differential systems}\label{sect-rest}
	
	In this section, we prove Theorem \ref{Serre} by utilizing Theorem \ref{jets N}. 
	
	\begin{proof}[Proof of Theorem \ref{Serre}]	
	We represent the Toda field theory \er{toda} as the Pfaffian system on the following infinite jet space with coordinates
	$$
	(x, y, u_i, u_{i,x}^{(k)}, u_{i, y}^{(k)}), \quad 1\leq i\leq \ell,\ k\geq 1.
	$$
	By convention, $u_i = u_{i, x}^{(0)} = u_{i, y}^{(0)}$, and $u_{i, x} = u_{i, x}^{(1)}$, $u_{i, y} = u_{i, y}^{(1)}$.  
	Let $\p_x$ and $\p_y$ be the derivative with respect to $x$ and $y$ respectively. We have $\p_x u_{i, x}^{(k)} = u_{i, x}^{(k+1)}$ and $\p_y u_{i, y}^{(k)} = u_{i, y}^{(k+1)}$. 
	Using the shorthand \er{rhoi},  the differential 1-forms defining the Toda field theory \er{toda} are 
	$$
	\begin{cases}
	du_i - u_{i,x} dx - u_{i, y}dy,\\
	du_{i,x}^{(k)} - u_{i, x}^{(k+1)}dx + \p_x^{k-1} e^{\rho_i} dy,\quad k\geq 1,\\
	du_{i,y}^{(k)} + \p_y^{k-1} e^{\rho_i} dx - u_{i, y}^{(k+1)} dy,\quad k\geq 1.
	\end{cases}
	$$
		
	Now we restrict $x$ to a constant in the jet space while the $u_{i, x}^{(k)}$ are still coordinates. Then $dx=0$ and the differential system becomes
	\begin{equation}\label{with x-jets}
	\begin{cases}
	du_{i,y}^{(k)} - u_{i, y}^{(k+1)}dy,\quad k\geq 0\\
	du_{i,x}^{(k)} + \p_x^{k-1} e^{\rho_i} dy,\quad k\geq 1.
	\end{cases}
	\end{equation}

	Define $B_i^k$ via 
	$$
	\p_x^k e^{\rho_i} = e^{\rho_i} B^k_i, \quad\text{for } 1\leq i\leq \ell,\ k\geq 0. 
	$$
	The $B^k_i$ are differential polynomials of the $u_j$ with respect to $x$ and they clearly satisfy that for $1\leq i\leq \ell$, 
	\begin{equation}\label{rec for B}
	B^0_i = 1, \qquad B^{k}_i = \p_x B^{k-1}_i + \rho_{i, x} B^{k-1}_i,\ k\geq 1.
	\end{equation}
	It is in this way that we make sense of the second set of 1-forms in \er{with x-jets}. 
	The first set of 
	1-forms in \er{with x-jets} are just some jet relations.
	Therefore disregarding them, we represent our Toda field theory \er{toda} restricted to $x=\text{constant}$ on the manifold with coordinates
	\begin{equation*}\label{res coor}
	(y, u_i, u_{i, x}^{(k)}), \quad 1\leq i\leq \ell,\ k\geq 1
	\end{equation*}
	by the system of 1-forms
	$$
	du_{i,x}^{(k)} + e^{\rho_i}B^{k-1}_i dy,\quad 1\leq i\leq \ell,\ k\geq 1.
	$$

	We use the dual viewpoint
	and choose the following basis of vector fields that generate at each point the corresponding distribution
	\begin{equation}\label{vf1}
	\begin{split}
	U_j &:= \sum_{i=1}^\ell a^{ij} \frac{\p}{\p u_i},\quad 1\leq j\leq \ell,\\
	Y &:= \frac{\p}{\p y} - \sum_{i=1}^\ell  e^{\rho_i} \sum_{k\geq 1} B^{k-1}_i \frac{\p}{\p u_{i, x}^{(k)}},
	\end{split}
	\end{equation}
	where $(a^{ij}) = (a_{ij})^{-1}$. 
	
	Taking Lie brackets and recalling \er{rhoi}, for $1\leq j \leq \ell$ we have
	\begin{equation}\label{def vj}
	\begin{split}
	V_j := [Y, U_j] =  \sum_{m, i=1}^\ell a^{mj} a_{im} e^{\rho_i} \sum_{k\geq 1} B^{k-1}_i \frac{\p}{\p u_{i, x}^{(k)}}
	 = e^{\rho_j} \sum_{k\geq 1} B^{k-1}_j \frac{\p}{\p u_{j, x}^{(k)}}.
	\end{split}
	\end{equation}
	Then $Y = \frac{\p}{\p y} - \sum_{j=1}^\ell V_j$, and the bracket relations so far are 
	\begin{equation}\label{more}
	[\frac{\p}{\p y}, U_j] = 0, \quad  [\frac{\p}{\p y}, V_j] = 0, \quad [U_i, V_j] = \delta_{ij} V_j.
	\end{equation}
	Now we study the bracket relations among the $V_j$, and we will show that they generate $\fn_-$. 
	
	Define	
	\begin{equation}\label{tvj}
	\wt{V_j} = e^{-\rho_j} V_j = \sum_{k\geq 1} B^{k-1}_j \frac{\p}{\p u_{j, x}^{(k)}}.
	\end{equation}
	Then for $1\leq j_1,\cdots,j_m\leq \ell$, we have
	$$
	[V_{j_1}, \cdots,[V_{j_{m-1}}, V_{j_m}]\cdots] = e^{\rho_{j_1} + \cdots + \rho_{j_m}} [\wt{V_{j_1}}, \cdots,[\wt{V_{j_{m-1}}}, \wt{V_{j_m}}]\cdots].
	$$
	Now we show that the $\wt{V_j}$ can be viewed as derivations on the ring $\bc[{\bf U}]/\ci$ from Theorem \ref{jets N}.
	
	First we note that the vector fields $Y$ and the $U_j$ in \er{vf1}
	annihilate
	all the characteristic integrals $I_i$ and their derivatives on the $x$-side. 
	This is the case for $Y$ by the definition of characteristic integrals and that $Y$ is the total derivative vector $\p_y$.
	This is the case for the $U_j$ since the characteristic integrals $I_i$ contain at least the first-order derivatives in view of \er{bring up} and the definition of 
	{\bf u} in \er{players}. 
	Therefore, by \er{def vj} and \er{tvj}, 
	 we have $\cl_{\wt{V_j}}\ci\subset \ci$ and the $\wt{V_j}$ 
	 descend as derivations on the ring $\bc[{\bf U}]/\ci$. 
	
	By the isomorphism in Theorem \ref{jets N}, derivations on $\bc[N_-]$ or equivalently (algebraic) vector fields on $N_-$ 
	canonically correspond to derivations on $\bc[{\bf U}]/\ci$. 
	For $1\leq j\leq \ell$, let $e_{-\a_j}^R$ be the vector field for the infinitesimal 
	action of $e_{-\a_j}$ under the right multiplication of $N_-$ on $N_-\subset G/B_+$.
	Explicitly,  we have
	\begin{equation}\label{eR}
	(e_{-\a_j}^R f)(n) = \frac{d}{dt}\Big|_{t=0} f(n \exp(te_{-\a_j})),\quad n\in N_-,\ f\in \bc[N_-].
	\end{equation}
	By \cite{K3}*{Proposition 3.5 and Theorem 2.2}, we know that as vector fields on $N_-$, 
	\begin{equation}\label{kos imp}
	[\cl_e, e_{-\a_j}^R] = -(\a_j, K^{-1} e K) e_{-\a_j}^R,
	\end{equation}
	where $K\in N_-$ and $(\a_j, K^{-1}eK)$ is a function on $N_-$. 
	
	We denote by $\wt{e_{-\a_j}^R}$ the derivation on $\bc[{\bf U}]/\ci$ corresponding to $e_{-\a_j}^R$ on $\bc[N_-]$.
	Note that $\cl_e$ corresponds to $\p_x$ under the isomorphism in Theorem \ref{jets N}. 
	By $\a_j = \sum_{m=1}^\ell a_{jm}\omega_m$ and in view of \er{coord} and \er{whichmap}, \er{kos imp} gives 
	\begin{equation}\label{new rel}
	[\wt{e_{-\a_j}^R}, \p_x] = \big(\sum_{m=1}^\ell a_{jm} u_{m, x}\big) \wt{e_{-\a_j}^R} = \rho_{j, x} \wt{e_{-\a_j}^R}.
	\end{equation}
	As a derivation on $\bc[{\bf U}]/\ci$, we have
	$$\p_x = \sum_{i=1}^\ell\sum_{k\geq 1} u_{i, x}^{(k+1)} \frac{\p}{\p u_{i, x}^{(k)}}.$$ 
	Write 
	$$\wt{ e_{-\a_j}^R} = \sum_{i=1}^\ell\sum_{k\geq 1} C^{k-1}_{j, i} \frac{\p}{\p u_{i,x}^{(k)}},$$ 
	where the $C^{k-1}_{j, i}$ are functions of $\{u_{i, x}^{(m)}\}_{1\leq i\leq \ell}^{m\geq 1}$. 
	Then from \er{new rel}, we see that
	$$
	C_{j, i}^{k} = \p_x C_{j, i}^{k-1} + \rho_{j, x} C_{j, i}^{k-1},\quad k\geq 1.
	$$
	Using \er{eR}, we see that
	$C_{j, i}^0 = \delta_{ij}$. 
	Hence $C_{j, i}^k = 0$ if $j\neq i$, and $C_{i, i}^k = B_i^k$ in view of \er{rec for B}. 
	Therefore, 
	\begin{equation}\label{identify}
	\wt{V_j} = \wt{e_{-\a_j}^R},
	\end{equation}
	as derivations on $\bc[{\bf U}]/\ci$. This shows, in particular, that the $\wt{V_j}$ and hence the $V_j$ generate $\fn_-$. 
	
	Now we show that the Lie algebra generated by the vector fields in \er{vf1} is isomorphic to 
	a Lie algebra of vector fields 
	on the first jet space 
	of the integral curves of the standard 
	differential system on the big cell $N_- \subset G/B_+$. 
	
	Two integral curves $\Phi_1, \Phi_2: I\to N_-$ of the standard differential system (see \er{stand sys}) are called 1-equivalent at a point $y_0\in I$ 
	if their graphs have a contact of order 1 at the point $\Phi_1(y_0) = \Phi_2(y_0)$. The equivalence class of $\Phi$ with respect to this relation 
	is denoted by $[\Phi]_{y_0}^1$ and is called the 1-jet of $\Phi$ at $y_0$. 
	For an interval $I\subset \br^1$, the first jet space is  
	$$J^1_y(N_-, D)=\{[\Phi]_{y_0}^1\,|\, \Phi:I\to N_-, y_0\in I, \Phi \text{ satisfies }		\er{stand sys}\}$$ 
	with a natural structure of manifold. 
	Let $\omega\in \Omega^1(N_-, \fn_-)$ be the Maurer-Cartan form. Then $\Phi$ 
	satisfies \er{stand sys} iff $\omega(\Phi'(y))\in \fg_{-1}$. 
	Hence we can write 
	\begin{equation}\label{sds out}
	\omega(\Phi'(y)) = \sum_{i=1}^\ell \phi_i(y) e_{-\a_i}
	\end{equation}
	using some functions $\phi_i(y)$. 
	
	Therefore, $J^1_y(N_-, D)$ is a manifold with coordinates 
	$$
	(y, \Phi^0, \phi_i^0),\quad 1\leq i\leq \ell,
	$$
	where $\Phi^0$ is the set of global coordinates on $N_-$.  
	Here the superscripts 0 stand for 0-jets (the map values), and are used to distinguish the coordinates on the jet space from the actual 
	representative maps. 
	By \er{sds out}, the Pfaffian system on $J_y^1(N_-, D)$ is defined by the components of the $\fg$-valued $1$-form
	\begin{equation}\label{ics}
	\omega  - \sum_{i=1}^\ell \phi_i^0 \,dy \otimes e_{-\a_i}.
	\end{equation}
	where $\omega$ is  expressed in terms of the coordinates $\Phi^0$. 
	
	Let us now consider the dual viewpoint in terms of distributions. We identify the tangent spaces of $N_-$, through the left translation to the identity, with the Lie algebra $\fn_-$. 
	Then $\{\frac{\p}{\p y}, \frac{\p}{\p \phi_i^0}, e_{-\a}\}$  is a basis for the tangent spaces to $J^1_y(N_-, D)$, where $1\leq i\leq \ell$ and the $-\a\in \Delta^-$ 
	are the negative roots. 
	The  distribution dual to \er{ics} is globally generated, where $\phi_i^0\neq 0$, by 
	\begin{gather*}
	\phi_i^0\frac{\p}{\p \phi_i^0},\quad 1\leq i\leq \ell,\quad\text{and}\\
	 \quad \ol Y = \frac{\p}{\p y} + \sum_{i=1}^\ell \phi_i^0 e_{-\a_i}.
	\end{gather*}
	since the Maurer-Cartan form satisfies $\omega(a) = a$ for $a\in \fn_-$.
	
	Analogously to \er{def vj}, define
	\begin{equation}\label{bvj}
	\ol{V_j} := [\ol Y, \phi_j^0 \frac{\p}{\p \phi_j^0}] = -\phi_j^0 e_{-\a_j},\quad 1\leq j\leq \ell.
	\end{equation}
	The $\ol{V_j}$ clearly generate $\fn_-$. We have $\ol Y = \frac{\p}{\p y} - \sum_{j=1}^\ell \ol{V_j}$ and also the bracket relations 
	\begin{equation}\label{the bars}
	[\frac{\p}{\p y}, \phi_i^0\frac{\p}{\p \phi_i^0}] = 0, \quad [\frac{\p}{\p y}, \ol{V_j}] = 0, \quad [\phi_i^0\frac{\p}{\p \phi_i^0}, \ol{V_j}] = \delta_{ij} \ol{V_j}.
	\end{equation}
	
	By the comparison of \er{more} with \er{the bars} and \er{tvj} with \er{bvj}, we see using \er{identify} that the map
	\begin{equation}\label{isom}
	Y\mapsto \ol Y,\quad U_i\mapsto \phi_i^0\frac{\p}{\p \phi_i^0},\ 1\leq i\leq \ell,
	\end{equation}
	establishes an isomorphism between the two sets of vector field generators of the two distributions which respects their successive Lie brackets. 
	This completes the proof of Theorem \ref{Serre}.
	\end{proof}
	
	\begin{example} We continue to use Example \ref{A_2} to illustrate our results. Using $e_{-\a_1} = E_{21}$, $e_{-\a_2} = E_{32}$ and the coordinates in
	\er{def K}, we see that 
	$$
	e_{-\a_1}^R = \frac{\p}{\p v_1} + v_2\frac{\p}{\p v_3}, \quad e_{-\a_2}^R = \frac{\p}{\p v_2}.
	$$
	By the transformation \er{transf} and the chain rule, we see that 
	$$
	\wt{e_{-\a_1}^R} = \frac{\p}{\p u_{1, x}}, \quad \wt{e_{-\a_2}^R} = \frac{\p}{\p u_{2, x}} + (2u_{2, x}-u_{1, x}) \frac{\p}{\p u_{2, xx}}.
	$$
	These are clearly the $\wt{V_1}$ and $\wt{V_2}$ from \er{tvj} expressed in the generators $u_{1, x}$, $u_{2, x}$ and $u_{2, xx}$ of $\bc[{\bf U}]/\ci$. 
	\end{example}

\section{From standard differential systems to Toda field theories}\label{sect-invs}
	
	
	In this section, we prove Theorem \ref{main}, 
	which realizes the quotient structure of \cite{AFV}*{Theorem 1.4} for Toda field theories as Darboux integrable differential systems. 
	
	
	\begin{proof}[Proof of Theorem \ref{main}]

	By \er{stand sys}, an integral curve in $N_-$ of the standard differential system is locally a map 
	$\Phi:  I \to N_-$ such that 
	\begin{equation}\label{def curve}
	a := \Phi^{-1} \Phi_y \in \fn_- \text{ lies in the subspace }\fg_{-1},
	\end{equation}
	where $\Phi_y = \Phi'(y)$. 
	
	Similarly there is the following ``transposed" version on the $x$-side. We have the similarly defined subgroups $N_+$ and $B_-$ 
	of $G$ integrating the Lie subalgebras $\fn_+ = \oplus_{i>0} \fg_i$ and $\fb_- = \oplus_{i\leq 0} \fg_i$. 
	Furthermore the space 
	$ B_-\backslash G$ of right cosets contains $N_+$ as a big cell. 
	The \emph{left} action of $G$ on $B_-\backslash G$ is defined by $g(B_- k) = B_- kg^{-1}$ for $g\in G$ and $B_-k\in B_-\backslash G$. 
	The action of $g$  in an open set of $G$ (which contains the identity)  on $n_2 \in N_+$ is defined by 
	\begin{equation}\label{hard-action}
		\text{if } n_2 g^{-1} = q\t n_2 \in B_- N_+, \text{ then } g\cdot n_2 := \t n_2.
	\end{equation}
		The standard distribution $\t D$ on $B_-\backslash G$ is defined as the $G$-invariant distribution equal to $\fg_1$ mod $\fb_-$ at $\t o = B_- \cdot 1$. 
	An integral curve in $N_+$ of the standard differential system is thus locally a map  $\Psi: J \to N_+$, where $J$ is an interval, such that
	\begin{equation}\label{def y-curve}
		b := \Psi_x \Psi^{-1} \in \fn_+\text{ lies in the subspace }\fg_{1},
	\end{equation}
		where $\Psi_x = \Psi'(x)$, $x\in J$. 
	
		We write the $a$ and $b$ in \er{def curve} and \er{def y-curve} as 
		\begin{equation}\label{ab}
		a=\sum_{i=1}^\ell \phi_i(y) e_{-\a_i}\quad\text{and}\quad b=\sum_{i=1}^\ell \psi_i(x) e_{\a_i}.
		\end{equation}  
		
		Therefore, we have the following coordinates 		
		$$
		(x, y, \Psi^0, \psi^0_i, \Phi^0, \phi^0_i),\quad 1\leq i\leq \ell.
		$$ 
		on the product $J^1_x(N_+, \t D)\times J^1_y(N_-, D)$ of jet spaces. 
		An element  $g\in G$ acts trivially on $x$ and $y$, acts on $\Phi^0$ and $\Psi^0$ by \er{action} and \er{hard-action} when defined, and acts on $\phi_i^0$ 
		and $\psi_i^0$ through the 
		prolonged action. 
		That is, given a point $(y_0, \Phi^0_0, \phi^0_{i, 0})\in J^1_y(N_-, D)$, we choose a representative map $\Phi: I\to N_-$ such that $\Phi(y_0) = \Phi^0_0$
		and 
		$\Phi(y_0)^{-1}
		\Phi_y(y_0) = \sum_{i=1}^\ell \phi^0_{i,0} e_{-\a_i}$, then we define
		$$
		\sum_{i=1}^\ell (g\cdot \phi^0_{i,0}) e_{-\a_i} := (g\cdot \Phi(y_0))^{-1}
		\frac{d}{dy}\Big|_{y=y_0}(g\cdot \Phi(y))
		$$
		where $g$ is close to the identity so that $g\cdot \Phi(y_0)$ is defined as in \er{action}. 
		The definition of $g\cdot \psi^0_i$ is similar. 
		We refer the reader to \cite{OP}*{Chapter 5} for background on prolonged group action and differential invariants.

		The number of coordinates that $G$ acts on is $\dim N_- + \dim N_+ + 2\ell = \dim G + \ell$. 
		Now we construct $\ell$ differential invariants and show that they satisfy the Toda field theory \er{toda}. 
		In this proof, we will work with representatives of the jet spaces.  
		That is, we will work with actual integral curves $\Phi: I\to N_-$ and $\Psi: J\to N_+$ satisfying \er{def curve} and \er{def y-curve}. 
		
		We use the bracket notation following the physicists \cite{LSbook}. For $1\leq i\leq \ell$, we denote 
		by $|i\ra$ the highest weight vector for the $i$th fundamental representation $V_{i}$ of $G$ \cite{FH}, and by $\la i |$ the lowest weight vector for the dual
		representation $V_i^*$, where $G$ acts from the right. There are the induced representations of $\fg$ and its universal enveloping algebra $U(\fg)$ (see \cite{Knapp}) on $V_{i}$ and $V_i^*$. 
		For $\mu, \nu\in U(\fg)$ and $g \in G$, $\la i |\nu g \mu| i \ra$ denotes the pairing of $\la i |\nu$ in $V_i^*$ with $g(\mu |i\ra)$ in $V_i$.
		We require that for the identity element $1\in G$, we have $\la i | 1 | i \ra = 1$. Then consider the following 
		function 
		$$
		\xi_i := \la i | \Psi(x) \Phi(y) |i\ra.
		$$
		This function is not invariant under the $G$-action, but its transformation under $g\in G$ is simple. 
		Using \er{action} and \er{hard-action}, we have 
		\begin{equation}
		g^{-1}\xi_i = \la i | (g\cdot \Psi(x)) (g\cdot \Phi(y)) |i\ra = \la i | Q(x)^{-1} \Psi(x) g^{-1} g \Phi(y) P(y)^{-1} | i \ra = \xi_i p(y) q(x).
		\end{equation}
		Here $Q(x) \in B_-$ and $P(y) \in B_+$ are group elements for the normalizations from \er{action} and \er{hard-action}. Since $\la i |$ and $|i\ra$ are  the lowest and highest weight vectors, we have 
		$$
		\la i | Q(x)^{-1} = q(x) \la i |, \quad P(y)^{-1}| i \ra = p(y) |i \ra,
		$$
		for some functions $p(y)$ and $q(x)$. 
		Therefore the locally defined
		$$
		-(\log \xi_i)_{xy}
		$$
		is $G$-invariant. 
		
		Now we define
		\begin{equation}\label{solns}
		u_i = -\log \xi_i + \sum_{j=1}^\ell a^{ij} \log (\phi_j(y) \psi_j(x)),\quad 1\leq i\leq \ell.
		\end{equation}
		The $u_{i, xy}=-(\log \xi_i)_{xy}$ are also differential invariants. 
		
		Now we compute that 
		$$
		u_{i, xy} = - \frac{\xi_i \xi_{i,{xy}} - \xi_{i, x} \xi_{i, y}}{\xi_i^2}.
		$$
		By \er{def curve}, \er{def y-curve} and \er{ab}, we have
		\begin{gather*}
		\xi_{i, y} = \la i | \Psi(x) \Phi_y(y) |i\ra = \la i | \Psi(x)  \Phi(y) \sum_{j=1}^\ell \phi_j(y) e_{-\a_j} | i \ra = \phi_i(y) \la i | \Psi(x)  \Phi(y) e_{-\a_i} | i \ra,\\
		\xi_{i, x} = \la i | \Psi_x(x) \Phi(y) |i\ra = \la i | \sum_{j=1}^\ell \psi_j(x) e_{\a_j} \Psi(x)  \Phi(y) | i \ra = \psi_i(x) \la i | e_{\a_i} \Psi(x)  \Phi(y) | i \ra,\\
		\xi_{i,{xy}} = \phi_i(y)\psi_i(x) \la i | e_{\a_i} \Psi(x)  \Phi(y) e_{-\a_i} | i \ra.
		\end{gather*}
		Here we have used the well-known fact that for $j\neq i$ we have $e_{-\a_j}|i\ra =0$ and similarly $\la i | e_{\a_j} = 0$ (see \cite{LSbook}*{Eq. (1.4.19)}). Therefore 
		$$
		u_{i, xy} = - \phi_i(y) \psi_i(x)\frac{\Delta_i}{\xi_i^2},
		$$
		where 
		$$
		\Delta_i = \begin{vmatrix}
		\la i | \Psi(x)  \Phi(y) | i \ra & \la i | \Psi(x)  \Phi(y) e_{-\a_i} | i \ra\\
		\la i | e_{\a_i} \Psi(x)  \Phi(y) | i \ra & \la i | e_{\a_i} \Psi(x)  \Phi(y) e_{-\a_i} | i \ra
		\end{vmatrix}.
		$$
		Now the so-called Jacobi identity \cite{LSbook}*{Eq. (1.6.16)} asserts that 
		\begin{equation}\label{Jacobi}
		\Delta_i = \prod_{j\neq i} \xi_j^{-a_{ij}}.
		\end{equation}
		
		Therefore since $a_{ii}=2$, we have
		\begin{equation*}\label{my toda}
		u_{i, xy} = - \phi_i(y) \psi_i(x) \prod_{j=1}^\ell \xi_j^{-a_{ij}} = -\exp\bigg(\sum_{j=1}^\ell a_{ij} u_j\bigg)
		\end{equation*}
		by the definition of $u_i$ in \er{solns}. 
		This shows that the $u_j$ for $1\leq j\leq \ell$ are functions in the $u_{i, xy}$ and hence are differential invariants themselves. 
		
		Therefore, the following functions 
		\begin{equation}\label{ui0}
		u_i^0 = -\log \la i | \Psi^0 \Phi^0 |i\ra + \sum_{j=1}^\ell a^{ij} \log (\phi_j^0 \psi_j^0),\quad 1\leq i\leq \ell,
		\end{equation}
		on $J^1_x(N_+,\t D)\times J^1_y(N_-, D)$ are invariant under the $G$ action, and they satisfy the Toda field theory \er{toda}.
		\end{proof}

	It is interesting to note that the above differential invariants in \er{solns} are exactly the general local solutions of the Toda field theory \er{toda} constructed in \cites{LS, LSbook}. 

	

\medskip
\end{document}